\documentclass{amsart}
\usepackage[utf8]{inputenc}
\usepackage{latexsym,amssymb,amsmath,amsthm,amscd,graphicx}
\usepackage{amsfonts}
\usepackage{xcolor}
\usepackage{tikz}
\usepackage[english]{babel}
\usepackage[a4paper,bindingoffset=0.2in,%
            left=1in,right=1in,top=1in,bottom=1in,%
            footskip=.25in]{geometry}
\newtheorem{theorem}{Theorem}[section]

\newtheorem{remark}[theorem]{Remark}
\newtheorem{cor}[theorem]{Corollary}

\numberwithin{equation}{section}

\begin{document}

\title[On the proper inclusion property of (discrete) Morrey spaces]{On the proper 
inclusion property of (discrete) Morrey spaces}

\author[Y.I~Runtunuwu]{Yohanes Imanuel Runtunuwu}
\address{Faculty of Mathematics and Natural Sciences, Bandung Institute of Technology, Bandung 40132, Indonesia}
\email{20122004@mahasiswa.itb.ac.id}

\author[J.~Lindiarni]{Janny Lindiarni}
\address{Faculty of Mathematics and Natural Sciences, Bandung Institute of Technology, Bandung 40132, Indonesia}
\email{janny@itb.ac.id}

\author[H.~Gunawan]{Hendra Gunawan}
\address{Faculty of Mathematics and Natural Sciences, Bandung Institute of Technology, Bandung 40132, Indonesia}
\email{hgunawan@itb.ac.id}



\begin{abstract}
In this paper, we construct a function which confirms the proper inclusion 
property of Morrey spaces, by using a relation between a class of functions in Morrey 
spaces and discrete Morrey spaces. Our particular function is simpler than those 
constructed by Gunawan \textit{et al.} in \cite{gunawan1, gunawan2}.
\end{abstract}

\maketitle

\section{Introduction}\label{intro}

Morrey spaces were first introduced by C.B. Morrey in 1938 in the study of the solution to
an elliptic partial differential equation \cite{morrey}. For $1\le p\le q<\infty$, the Morrey 
space $\mathcal{M}^p_q=\mathcal{M}^p_q(\mathbb{R}^d)$ consists of all functions $f\in 
L_{loc}^p(\mathbb{R}^d)$ such that 
$$
\left\|f\right\|_{\mathcal{M}_{q}^{p}}:=\sup_{a\in\mathbb{R}^d,r>0}|B(a,r)|^{\frac{1}{q}-\frac{1}{p}}\left(\int_{B(a,r)}|f(x)|^p dx\right)^{\frac{1}{p}}<\infty.
$$
As shown in \cite{sawano}, Morrey spaces satisfy the inclusion property: 
for $1\le p_1\le p_2\le q<\infty$, we have $\mathcal{M}_{q}^{p_2} \subseteq 
\mathcal{M}_{q}^{p_1}$, with $\|f\|_{\mathcal{M}_q^{p_1}}\le\|f\|_{\mathcal{M}_q^{p_2}}$ 
whenever $f\in \mathcal{M}_q^{p_2}$.

In \cite{gunawan1}, the inclusion property is shown to be proper for $1\le p_1<p_2\le q< 
\infty$, by constructing a function $f$ belonging to $\mathcal{M}_q^{p_1}$ which is not in 
$\mathcal{M}_q^{p_2}$. However, the function is constructed in a subtle way, by modifying the 
properties of power functions. For $d=1$, a different function which confirms the proper 
inclusion property of Morrey spaces may be obtained indirectly via the proper inclusion 
property of discrete Morrey spaces, which we shall discuss later.

In this paper, we shall confirm that, for $1\le p_1<p_2\le q<\infty$, the inclusion 
$\mathcal{M}^{p_2}_q(\mathbb{R}) \subset\mathcal{M}^{p_1}_q(\mathbb{R})$ is proper by 
constructing a simpler function belonging to $\mathcal{M}_q^{p_1}(\mathbb{R}) 
\setminus \mathcal{M}_q^{p_2}(\mathbb{R})$. We achieve this by using relation between 
a class of functions in Morrey spaces and discrete Morrey spaces.

Note that throughout the paper, the letters $C$ and $C_{p,q}$ denote some constants, which 
may have different values from line to line.

\section{Recent Results}

For $1\le p\le q<\infty$, the discrete Morrey Space $\ell_q^p=\ell_q^p(\mathbb{Z})$, which
was first introduced in \cite{gunawan3}, is defined to be the set of all sequences 
$x=(x_j)_{j\in\mathbb{Z}}$ such that
\begin{equation}\label{discrete1}
\|x\|_{\ell_q^p}:=\sup_{m\in\mathbb{Z},N\in\mathbb{N}_0}|S_{m,N}|^{\frac{1}{q}-\frac{1}{p}}
\left(\sum_{j\in S_{m,N}}|x_j|^p\right)^{\frac{1}{p}}<\infty,
\end{equation}
where $\mathbb{N}_0:=\mathbb{N}\cup \{0\},S_{m,N}:=\{m-N,...,m,...,m+N\},$ and 
$|S_{m,N}|=2N+1$. Notice that a sequence $x\in\ell^p_q$ must be bounded with $\|x\|_\infty
\le \|x\|_{\ell^p_q}$, for $1\le p\le q<\infty$. As shown in \cite{gunawan3}, 
discrete Morrey spaces also satisfy the inclusion property: for $1\le p_1\le p_2\le q<
\infty$, we have $\ell^{p_2}_q\subseteq \ell^{p_1}_q$, with $\|x\|_{\ell^{p_1}_q}\le 
\|x\|_{\ell^{p_2}_q}$ whenever $x\in\ell^{p_2}_q$. 

Later, it was shown in \cite{gunawan2} that, for $1\le p_1<p_2\le 
q<\infty$, the inclusion $\ell^{p_2}_q\subset \ell^{p_1}_q$ is proper by constructing a 
particular sequence in $\ell^{p_1}_q\setminus \ell^{p_2}_q$. To be precise, consider
the sequence $x=(x_j)_{j\in\mathbb{Z}}$ which is given by:
\begin{equation}\label{seq}
x_j:=\begin{cases}
        1,\text{ if $|j|=0,1,2,...,2^{w+v}$},\\
        1,\text{ if $|j|=2^{k(w+v)},2^{k(w+v)}-2^{kw},...,2^{k(w+v)}-2^{k(w+v-2)}$},\\
        0,\text{ otherwise,}
    \end{cases}
\end{equation}
where $k\in\mathbb{N}\cap[k_0,\infty), k_0$ be the smallest positive integer such that 
$1-\frac{1}{2^{2k_0}}>\frac{1}{2^{w+v-1}}$ and $v,w\in\mathbb{N}$ chosen such that
$$
\left(\frac{q}{p_2}-1\right)w+\frac{2q}{p_2}<v<\left(\frac{q}{p_1}-1\right)w+2.
$$ 
One may then observe that $x\in \ell_q^{p_1} \setminus \ell_q^{p_2}$ for $1\le p_1<p_2\le 
q<\infty$.

Using this result for discrete Morrey spaces, one can obtain an alternative function that
confirms the proper inclusion $\mathcal{M}^{p_2}_q(\mathbb{R})\subset 
\mathcal{M}^{p_1}_q(\mathbb{R})$ for $1\le p_1<p_2\le q<\infty$ via the following 
observation. For $x:=(x_j)_{j\in\mathbb{Z}}\in \ell^p_q$, let
\begin{equation}\label{embeding}
\tilde{x}(t):=\sum_{j\in\mathbb{Z}} x_j\chi_{[j,j+1)}(t),
\end{equation}
which defines a function on the real line. Then, we have the following result.

\begin{theorem}\label{equivalence}{\rm \cite{rizma}}
For $1\le p\le q<\infty$, the space $\ell^p_q$ may be embedded into the space 
$\mathcal{M}^p_q$ via the mapping (\ref{embeding}), where
$$
\|x\|_{\ell^p_q}\le \|\tilde{x}\|_{\mathcal{M}^p_q}\le C_{p,q}\|x\|_{\ell^p_q},
$$
with $C_{p,q}=5^{1/p-1/q}$.
\end{theorem}

\begin{cor} \label{Cor}
For $1\le p_1<p_2<q<\infty$, let $\widetilde{x}(t)=\sum_{j\in\mathbb{Z}}x_j\chi_{[j,j+1)}(t)$
where $x=(x_j)_{j\in\mathbb{Z}}\in \ell^{p_1}_q\setminus \ell^{p_2}_q$. Then $\widetilde{x}\in\mathcal{M}_q^{p_1}\setminus\mathcal{M}_q^{p_2}.$
\end{cor}

\begin{proof}
For a sequence $x:=(x_j)_{j\in\mathbb{Z}}\in\ell_q^{p_1}\setminus \ell_q^{p_2}$ where $1\le p_1<p_2<q<\infty$, it follows from Theorem \ref{equivalence} that 
$$
\|\widetilde{x}\|_{\mathcal{M}_q^{p_2}}  \ge\|x\|_{\ell_q^{p_2}}=\infty,
$$
which means that $\widetilde{x}\notin \mathcal{M}_q^{p_2}$, and 
$$
\|\widetilde{x}\|_{\mathcal{M}_q^{p_1}} \le C\|x\|_{\ell_q^{p_1}}<\infty,
$$
which means that $\widetilde{x}\in \mathcal{M}_q^{p_1}$. 
Thus, $\widetilde{x}\in\mathcal{M}_q^{p_1}\setminus\mathcal{M}_q^{p_2},$ as claimed.
\end{proof}

Now, as one would realize, the space $\ell^p_q$ with the norm defined in (\ref{discrete1}) is
rather loose since the sets $S_{m,N}$ have always an odd cardinality. In this paper, we 
shall define $\ell_q^p$, for $1\le p\le q<\infty$,  to be the set of all sequences $x=
(x_j)_{j\in\mathbb{Z}}$ for which
\begin{equation}\label{discrete2}
\|x\|^*_{\ell_q^p}
    :=\sup_{k\in\mathbb{Z},n\in N_0} |S^*_{k,n}|^{\frac{1}{q}-\frac{1}{p}}\left(\sum_{j\in S^*_{k,n}}|x_j|^p\right)^{\frac{1}{p}}<\infty,
\end{equation}
where $\mathbb{N}_0:=\mathbb{N}\cup\{0\},S^*_{k,n}:=\{k,...,k+n\},$ and $|S^*_{k,n}|=n+1$,
for $k\in\mathbb{Z}$ and $n\in\mathbb{N}_0$. 
Here $S^*_{k,n}$ may have an odd or even cardinality.

It is clear that the earlier norm $\|\cdot\|_{\ell^p_q}$ is dominated by the later 
$\|\cdot\|^*_{\ell^p_q}$. Conversely, one can find a constant $C>1$ such that 
$\|\cdot\|^*_{\ell^p_q}\le C\|\cdot\|_{\ell^p_q}$. Precisely, we have the following 
theorem.

\begin{theorem} \label{teo1} \cite{iconmaa}
Let $1\le p\le q<\infty$. Then, the following relations  
$$
\|x\|_{\ell_q^p}\le\|x\|^*_{\ell_q^p}\le C_{p,q}\|x\|_{\ell_q^p},\quad x\in\ell^p_q,
$$
hold for $C_{p,q}=\left(\frac{3}{2}\right)^{1/p-1/q}$.
\end{theorem}

Using the new norm $\|\cdot\|^*_{\ell^p_q}$, we have the following result on discrete
Morrey spaces.

\begin{theorem} \label{teo2} \cite{iconmaa}
For $1\le p\le q<\infty,$ the space $(\ell^p_q,\|\cdot\|^*_{\ell^p_q})$ may be embedded 
into the space $\mathcal{M}^p_q$ via the mapping (\ref{embeding}), where
$$
\|x\|^*_{\ell_q^p}\le\|\widetilde{x}\|_{\mathcal{M}_q^p}\le C_{p,q}\|x\|^*_{\ell_q^p},
$$
with $C_{p,q}=2^{1/p-1/q}$.
\end{theorem}

As the inequality $\|x\|^*_{\ell^p_q}\le\|\widetilde{x}\|_{\mathcal{M}^p_q}$ is
immediate, the proof of the above theorem focuses on establishing the inequality 
$\|\widetilde{x}\|_{\mathcal{M}_q^p}\le C_{p,q}\|x\|^*_{\ell_q^p}$. This can be
achieved by estimating the values of $(2r)^{1/q-1/p}\bigl(\int_{a-r}^{a+r}
|\widetilde{x}(t)|^p\,dt\bigr)^{1/p}$ for the cases where $0<r\le\frac12$, 
$\frac12<r\le1$, and $r>1$, wherever $a\in\mathbb{R}$.

As a corollary of Theorems \ref{teo1} and \ref{teo2}, we have the following refinement 
of Theorem \ref{equivalence} for the space $(\ell^p_q,\|\cdot\|_{\ell^p_q})$.

\begin{cor}
For $1\le p\le q<\infty$, the space $(\ell^p_q,\|\cdot\|_{\ell^p_q})$ may be embedded 
into the space $\mathcal{M}^p_q$ via the mapping (\ref{embeding}), where
$$
\|x\|_{\ell^p_q}\le \|\tilde{x}\|_{\mathcal{M}^p_q}\le C_{p,q}\|x\|_{\ell^p_q},
$$
with $C_{p,q}=3^{1/p-1/q}$.
\end{cor}

\section{A New Example}

In this section, we shall reconfirm that, for $1\le p_1<p_2\le q<\infty$, the inclusion 
$\mathcal{M}^{p_2}_q(\mathbb{R}) \subset\mathcal{M}^{p_1}_q(\mathbb{R})$ is proper by 
constructing a simpler function of the form (\ref{embeding}) which belongs to 
$\mathcal{M}_q^{p_1}(\mathbb{R}) \setminus \mathcal{M}_q^{p_2}(\mathbb{R})$. 
We note that, in the sequence $(x_j)_{j\in\mathbb{Z}}$ defined by (\ref{seq}), the values 
of $1$ within $[0,2^{w+v}]$ have a distance of $1$, and the values of $1$ within 
$[2^{k(w+v)}-2^{k(w+v-2)},2^{k(w+v)}]$ have a distance of $2^{kw}$. There are three
parameters involved, namely $k_0, v,$ and $w$. Not only the definition 
is difficult to read, the choice of $k_0,v,w\in\mathbb{N}$ is also not intuitive. 

As an alternative, we will construct a simpler and more intuitive sequence than the 
sequence (\ref{seq}), which also confirms that the inclusion $\ell^{p_2}_q\subset 
\ell^{p_1}_q$ is proper, where these spaces are equipped with the norm 
$\|\cdot\|^*_{\ell^{p_i}_q}$ defined in (\ref{discrete2}). Precisely, 
we have the following result.

\begin{theorem} \label{teo3}
Let $1\le p_1<p_2<q<\infty$. Choose $v, w\in\mathbb{N}$ such that
$$
\frac{q}{p_2}<\frac{v}{w}<\frac{q}{p_1}.
$$
Let $S_n:=\{1+2^v+...+2^{(n-1)v}+j.2^{nw}:j=1,...,2^{n(v-w)}\}$. Define 
$x=(x_j)_{j\in\mathbb{Z}}$ with
\begin{equation}\label{new}
        x_j:=\begin{cases}
            1,\text{ if}\ j\in\{0,1\}\cup \bigcup_{n=1}^\infty S_n,\\
        0,\text{ otherwise.}
        \end{cases}
\end{equation}
Then $x\in \ell_q^{p_1}\setminus \ell_q^{p_2}$.
\end{theorem}

Before we give the proof, let us first explain the choice of $v$ and $w$ in the above 
sequence. As $\frac{q}{p_2}<\frac{q}{p_1}$, we can obviously choose a rational number 
$\frac{v}{w}>1$ between the two numbers, with $v,w\in\mathbb{N}$. We note that within the 
interval $[1+2^v+...+2^{(n-1)v}+2^{nw}, 1+2^v+...+2^{(n-1)v}+2^{nv}]$ there are 
exactly $2^{n(v-w)}$ terms that take the value of $1$ and these values of $1$ have a 
distance of $2^{nw}$ from one to another, for every $n\in\mathbb{N}$. Moreover, there is
a gap of $2^{(n+1)w}$ between the last number in $S_n$ and the first number in $S_{n+1}$.
This tells us that the terms $x_j$'s which have the value of $1$ in this sequence are 
getting more and more distant from one to another as the index $j$ increases.
From the way it is constructed, this sequence is simpler and more intuitive than the 
sequence in \cite{gunawan3}.

\medskip

We now give the proof of Theorem \ref{teo3}.

\begin{proof}
Since the terms $x_j$'s having the value of $1$ in the sequence are more and more distant 
from one to another as the index $j$ increases, we only consider the sets $S^*_{0,k}$ where
$k\in\mathbb{N}_0$. Let $n\in\mathbb{N}_0$ and $1\le p<\infty$. Write
$$
\beta_n:=1+2^v+2^{2v}...+2^{nv}
$$
and
$$
\sigma_n:=\sum_{j\in S^*_{0,\beta_n}}|x_j|^p
        =1+2^{v-w}+2^{2(v-w)}+...+2^{n(v-w)}.
$$
Clearly $\beta_n>2^{nv}$ and $\sigma_n>2^{n(v-w)}$. On the other hand, we have
\begin{align*}
\beta_n=1+2^v+...+2^{(n-1)v}+2^{nv}=\frac{2^{nv}-1}{2^{v}-1}+2^{nv}<2.2^{nv},
\end{align*}
and
\begin{align*}
        \sigma_n
        =1+2^{v-w}+...+2^{(n-1)(v-w)}+2^{n(v-w)}=\frac{2^{n(v-w)}-1}{2^{v-w}-1}+2^{n(v-w)}<2.2^{n(v-w)}.
    \end{align*}
Hence, for $p=p_2$, we get
    \begin{align*}
        |S^*_{0,\beta_n}|^{\frac{1}{q}-\frac{1}{p_2}}
        \left(\sum_{j\in S^*_{0,\beta_n}}|x_j|^{p_2}\right)^{\frac{1}{p_2}}
        &> (2.2^{nv}+1)^{\frac{1}{q}-\frac{1}{p_2}}(2^{n(v-w)})^{\frac{1}{p_2}}\\
        &>(3.2^{nv})^{\frac{1}{q}-\frac{1}{p_2}}(2^{n(v-w)})^{\frac{1}{p_2}}\\
        &=3^{\frac{1}{q}-\frac{1}{p_2}} 2^{n\left(\frac{v}{q}-\frac{w}{p_2}\right)}.
    \end{align*}
By the values of $v$ and $w$ that we choose, we have $\frac{v}{q}-\frac{w}{p_2}>0$. 
As a consequence, $2^{n\left(\frac{v}{q}-\frac{w}{p_2}\right)}\to\infty$ when 
$n\to\infty$. Thus, $\|x\|_{\ell_q^{p_2}}=\infty$, which means that $x\notin \ell_q^{p_2}$.

We will now show that $x\in\ell_{q}^{p_1}$. First, we see that, for $p=p_1$ and $0\le m\le 1+2^v$ we have
$$|S^*_{0,m}|^{\frac{1}{q}-\frac{1}{p_1}}\left(\sum_{j\in S^*_{0,m}} 
|x_j|^{p_1}\right)^{\frac{1}{p_1}}\le|S^*_{0,0}|^{\frac{1}{q}-\frac{1}{p_1}}
\left(\sum_{j\in S^*_{0,1+2^v}} |x_j|^{p_1}\right)^{\frac{1}{p_1}}=1+2^{v-w}.$$
On the other hand, if $m>1+2^v$, then there exists $n\in\mathbb{N}$ such that $\beta_n\le m<\beta_{n+1}$. For such $m'$s, we observe that
\begin{align*}
    |S^*_{0,m}|^{\frac{1}{q}-\frac{1}{p_1}}\left(\sum_{j\in S^*_{0,m}} 
|x_j|^{p_1}\right)^{\frac{1}{p_1}}&\le |S^*_{0,\beta_n}|^{\frac{1}{q}-
\frac{1}{p_1}}\left(\sum_{j\in S^*_{0,\beta_{n+1}}} |x_j|^{p_1}\right)^{\frac{1}{p_1}}\\
&\le (2^{nv}+1)^{\frac{1}{q}-\frac{1}{p_1}} (2.2^{(n+1)(v-w)})^{\frac{1}{p_1}}\\
&\le 2^{\frac{1}{p_1}}(2^{nv})^{\frac{1}{q}-\frac{1}{p_1}}(2^{(n+1)(v-w)})^{\frac{1}{p_1}}\\
&=2^{\frac{1+v-w}{p_1}}2^{n\left(\frac{v}{q}-\frac{w}{p_1}\right)}\\
&\le 2^{\frac{1+v-w}{p_1}}.
\end{align*}
because $\frac{v}{q}-\frac{w}{p_1}<0$ by the choice of $v$ and $w$. By combining the two observations, we obtain $\|x\|_{\ell_{q}^{p_1}}<\infty$ or $x\in \ell_{q}^{p_1}$.
    
It follows from this together with the previous computation that $x\in \ell_{q}^{p_1} 
\setminus \ell_{q}^{p_2}$, as desired.
\end{proof}

\section{Concluding Remarks}

Speaking about discrete Morrey spaces, we have the following result on the inclusion 
properties.

\begin{theorem}\cite[Theorem 3.1]{HS}
    Let $1\le p_j\le q_j<\infty$ where $j=1,2$. Then the inclusion $\ell^{p_2}_{q_2}
    \subseteq \ell^{p_1}_{q_1}$ holds if and only if $q_2\le q_1$ and
    $\frac{p_1}{q_1}\le\frac{p_2}{q_2}$.
\end{theorem}

We shall argue here that the second part of the necessary condition can be established
by invoking the function we constructed in the previous section. To be precise, we can
prove that if $\ell^{p_2}_{q_2} \subseteq \ell^{p_1}_{q_1}$, then $\frac{p_1}
{q_1}\le\frac{p_2}{q_2}$. Indeed, assume to the contrary that $\frac{p_1}{q_1}>
\frac{p_2}{q_2}$. Choose $v,w \in \mathbb{N}$ such that $\frac{q_1}{p_1}<\frac{v}{w}<
\frac{q_2}{p_2}$, and define the sets $S_n$'s and the sequence $x=(x_j)_{j\in\mathbb{Z}}$ 
as in (\ref{new}). Then one may observe that $x\in \ell^{p_2}_{q_2}$ but $x\notin 
\ell^{p_1}_{q_1}$. This contradicts the hypothesis that $\ell^{p_2}_{q_2} \subseteq 
\ell^{p_1}_{q_1}$.

\begin{remark} 
From the above arguement, we note that if $1\le p_j\le q_j <\infty$ where $j=1,2$, $q_2\le 
q_1$, and $\frac{p_1}{q_1}<\frac{p_2}{q_2}$, then the inclusion $\ell^{p_2}_{q_2}\subseteq 
\ell^{p_1}_{q_1}$ is proper.
\end{remark}

\end{document}